\documentclass{amsart}
\usepackage[T1]{fontenc}
\usepackage{amssymb, amsmath, color, textcomp, url,tikz}
\usetikzlibrary{matrix,arrows}
\usetikzlibrary{fit}
\usetikzlibrary{positioning}
\usepackage[labelformat=empty]{caption}

\usepackage{soul}

\usepackage{mdwlist}

\RequirePackage{ifpdf}
\ifpdf
   \usepackage[pdftex]{hyperref}
\else
   \usepackage[hypertex]{hyperref}
\fi

\usepackage{enumerate,xspace}

\theoremstyle{plain}
\newtheorem{theorem}{Theorem}[section]
\newtheorem{cor}[theorem]{Corollary}
\newtheorem{prop}[theorem]{Proposition}
\newtheorem{lemma}[theorem]{Lemma}

\newtheorem*{teoA}{Theorem A}
\newtheorem*{teoB}{Theorem B}

\theoremstyle{definition}
\newtheorem{remark}[theorem]{Remark}
\newtheorem{fact}[theorem]{Fact}
\newtheorem{definition}[theorem]{Definition}
\newtheorem{example}[theorem]{Example}

\newtheorem*{ackno}{Acknowledgements}

\newtheorem*{claimstar}{Claim}
\newenvironment{claimstarpf}{\noindent\textit{Proof of
		Claim.}}{\hfill\qedsymbol \tiny{Claim}
	\medskip}

\newcommand{\nc}{\newcommand}

\nc{\Z}{\mathbb{Z}}
\nc{\Q}{\mathbb{Q}}
\nc{\N}{\mathbb{N}}
\nc{\F}{\mathbb{F}}
\nc{\UU}{\mathbb{U}}
\nc{\C}{\mathbb{C}}

\nc{\M}{\mathcal{M}}
\nc\LL{\mathcal L}

\nc{\dcl}{\operatorname{dcl}}
\nc{\dclq}{\operatorname{acl^\text{eq}}}

\nc{\acl}{\operatorname{acl}}
\nc{\aclq}{\operatorname{acl^\text{eq}}}
\nc{\nf}[1]{_{\mid {#1}}}
\nc{\restr}[1]{\!\!\upharpoonright_{#1}}
\nc{\sbgp}[1]{\langle\xspace {#1}\xspace\rangle}

\nc{\stab}{\operatorname{Stab}}
\nc\CAN{\operatorname{CB}}
\nc\inv{ ^{-1}}

\nc{\tp}{\operatorname{tp}}
\nc\cb{\operatorname{Cb}}
\nc\U{\operatorname{U}}
\nc{\cf}{\text{cf. }}
\nc{\eg}{\text{e.g. }}

\def\Ind#1#2{#1\setbox0=\hbox{$#1x$}\kern\wd0\hbox to
  0pt{\hss$#1\mid$\hss} \lower.9\ht0\hbox to
  0pt{\hss$#1\smile$\hss}\kern\wd0}
\def\Notind#1#2{#1\setbox0=\hbox{$#1x$}\kern\wd0\hbox to
  0pt{\mathchardef\nn="0236\hss$#1\nn$\kern1.4\wd0\hss}\hbox to
  0pt{\hss$#1\mid$\hss}\lower.9\ht0 \hbox to
  0pt{\hss$#1\smile$\hss}\kern\wd0}

\def\indip{\mathop{\ \ \hbox to 0pt{\hss$\mid^{\hbox to
0pt{$\scriptstyle P$\hss}}$\hss}
\lower4pt\hbox to 0pt{\hss$\smile$\hss}\ \ }}
\def\nindip{\mathop{\ \ \hbox to 0pt{\hss$\!\not{\mid}^{\hbox to
0pt{$\scriptstyle\, P$\hss}}$\hss}
\lower4pt\hbox to 0pt{\hss$\smile$\hss}\ \ }}

\begin{document}

\title{A model-theoretic note on the Freiman-Ruzsa theorem}
\date{\today}

\author{Amador Martin-Pizarro, Daniel Palacin and Julia Wolf}
\address{Abteilung f\"ur Mathematische Logik, Mathematisches Institut,
  Albert-Ludwig-Universit\"at Freiburg, Ernst-Zermelo-Stra\ss e 1, D-79104
  Freiburg, Germany}
\address{Department of Pure Mathematics and Mathematical Statistics, Centre for
Mathematical Sciences, Wilberforce Road, Cambridge CB3 0WB, United Kingdom}

\email{pizarro@math.uni-freiburg.de}
\email{palacin@math.uni-freiburg.de}
\email{julia.wolf@dpmms.cam.ac.uk}
\thanks{The first two authors conducted research partially supported by MTM2017-86777-P as well as by the Deutsche
	Forschungsgemeinschaft (DFG, German Research Foundation) - 
	Project number 2100310201 and 2100310301, part of the ANR-DFG 
	program GeoMod}
\keywords{Model Theory, Local Stability, Additive Combinatorics, Freiman-Ruzsa}
\subjclass{03C13, 03C45, 11B30}

\begin{abstract}
 A non-quantitative version of the Freiman-Ruzsa theorem is obtained for finite stable sets
 with small tripling in arbitrary groups, as well as for (finite) weakly normal subsets in
 abelian groups.
 \end{abstract}

\maketitle

\section*{Introduction}

A finite subset $A$ of a group $G$ is said to have \emph{doubling $K$} if the product set $A\cdot A=\{a\cdot b\,|\, a, b \in A\}$ has
size at most $K|A|$. Archetypal examples of sets with
small
doubling (where $K$ is constant as the size of the group $G$, and the set $A$, tend to infinity) are cosets of
subgroups. Theorems of Freiman-Ruzsa type assert that sets with small doubling are ``not too far'' from being subgroups in a suitable sense. Specifically, Freiman's original theorem \cite{gF73} asserts that a finite subset of
the integers with small doubling is efficiently contained in a generalized arithmetic progression. A proof of an analogous statement for arbitrary abelian groups was given by Green and Ruzsa \cite{GR07}, based on Ruzsa's proof of Freiman's theorem \cite{iR94}. A version of the result for abelian groups of bounded exponent with a particularly pleasing proof was given by Ruzsa in \cite{iR99}. His result asserts that if $A$ is a finite subset of an abelian group $(G,+)$ of
exponent $r$ such that $|A+A|\leq K|A|$, then $A$ is contained in a subgroup
$H$ of $G$ of size at most $K^2 r^{K^4}|A|$. For $G=\mathbb{F}_p^n$ with $p$ a fixed
prime, the exponent can be improved to $2K-1$ (see \cite{EL14} and references
therein). By considering the union of a subspace and $K$ arbitrary
linearly independent elements, it is not difficult to see that any bound on the
size of a subgroup containing $A$ must be exponential in $K$.

However, this example is still highly structured in the sense that a large part of the set has the structure of a subgroup, which suggests a natural reformulation of the problem. The \emph{Polynomial Freiman-Ruzsa Conjecture}, which remains one of the central open problems in additive combinatorics, asserts that a subset $A$ of doubling $K$ in $\mathbb
F_2^\infty$ can be covered by $C_1(K)$ many cosets of some subspace of size
$C_2(K)|A|$, where both $C_1(K)$ and $C_2(K)$ are polynomials in $K$; or equivalently, that there are constants
$C_3(K)$ and $C_4(K)$, each polynomial in $K$, such that for some  coset $v+H$ of a subspace $H$ of size
$C_3(K)|A|$, we have that $A\cap (v+H)$ has size at least
$|A|/C_4(K)$. For the best bounds known to date see \cite{tS12,tSch11}.

The above formulation of the Freiman-Ruzsa theorem resonates with a classical
setting in model theory, namely weakly normal groups. Weakly normal groups, also known as $1$-based
stable groups, are groups for which every definable set is a boolean combination of instances of weakly normal
formulae (see Section \ref{S:wn}). In a weakly normal
(stable) group, every definable subset is a boolean combination of cosets of definable subgroups \cite{HP87}.
Furthermore, every type over a model is the generic type of a coset of a
(type-)definable subgroup: the subgroup is its model-theoretic stabiliser.
Roughly speaking, a large proportion of a given definable set intersects a
coset of a definable group, so they are \emph{commensurable}.

For non-abelian groups, the suitable notion of doubling is \emph{tripling} $K$, that is, the cardinality of $A\cdot A \cdot A$ is bounded by $K|A|$. Indeed, sets of small tripling have small doubling, but the converse need not hold.  In this context, phenomena of Freiman-Ruzsa type are present in recent work of Hrushovski
\cite[Corollary 4.18]{eH12}, who
showed that a set of small tripling in a (possibly infinite) group of bounded
exponent is commensurable  with a subgroup, inspired by classical results and
techniques from stability theory in a non-standard setting.

Motivated by Hrushovski's work, in this note we adapt the local approach to stability of Hrushovski and
 Pillay in \cite[Theorem 4.1]{HP94} in order to
obtain a non-quantitative version of the Freiman-Ruzsa theorem for arbitrary
(possibly infinite) groups under the assumption of stability. We say that
a subset $A$ of $G$ is \emph{$r$-stable} if there are no elements
$a_1,\ldots, a_r, b_1, \ldots, b_r$ in $G$ such that $b_j\cdot a_i$ belongs to
$A$ if and only if  $i\leq j$. In particular, we
prove the following result.

\begin{teoA}\label{thm:teoA}
Given real numbers $K\geq 1$ and $\epsilon>0$ and a natural number $r\geq 
2$, 
there
exists a natural number $n=n(K, \epsilon, r)$ such that
for any (possibly infinite) group $G$ and any finite $r$-stable subset 
$A\subseteq G$ with tripling $K$, there is a
subgroup $H\subseteq A\cdot A\inv$
of $G$ with $A\subseteq C\cdot H$ for some $C\subseteq A$ of size at most $n$. Moreover, there exists $C'\subseteq C$ such that 
\[ |A \triangle (C'\cdot H)|\leq \epsilon |H|.\]
\noindent In particular, it follows from the Pl\"unnecke-Ruzsa inequalities that 
\[ |A \triangle (C'\cdot H)|\leq \epsilon K^2 |A|.\]
\end{teoA}

In the case when $G$ is abelian, it suffices to assume that $A$ has doubling
	$K$. Furthermore, we shall prove that, when $G$ is abelian, the subgroup
	$H$ can be taken to be a boolean
	combination (of complexity only depending on $K$, $\epsilon$ and $r$) of 
	translates of $A$. 
	
In particular, on choosing $\epsilon=1$, Theorem A implies that there is some 
natural number 
$n_0=n(K, 1, r)$ such that any finite $r$-stable subset $A$ of tripling $K$ is 
contained in $n_0$ translates of a subgroup $H\subseteq A\cdot A\inv$. It follows that $|A\cap g\cdot H|\geq 
|A|/n_0$ 
for some subgroup $H\leqslant G$ and some $g\in A$, which is a qualitative 
result in the spirit of Freiman-Ruzsa. As in the previous paragraph, when $G$ 
is abelian, the complexity of such a subgroup $H$ as a boolean combination of 
translates of $A$ can be bounded solely in terms of $K$ and $r$.

The above result dovetails with a suite of 
arithmetic regularity lemmas under the additional assumption of stability that 
have been obtained recently by Terry and the third author \cite{TW18, TW19}, as 
well as by Conant, Pillay and Terry \cite{CPT19}. However, without the 
assumption of small doubling/tripling, the bound on the symmetric difference is 
at best $\epsilon |H|$. Furthermore, the group $H$ so obtained in \cite{TW18, TW19,CPT19} has finite index in $G$ so its size comparable to $|G|$, but not necessarily to $|A|$.  The Theorem A is also reminiscent of work of Sisask \cite[Theorem 
5.4]{oS18}, who combined the assumption of small doubling with that of bounded 
VC-dimension in vector spaces over finite fields. Finally, we remark that closely related results were obtained by 
Conant \cite[Corollary 1.4]{gC20} for groups of bounded exponent. In a previous version of this article, Theorem A was stated with the upper bound  \[ |A \triangle (C'\cdot H)|\leq \epsilon |A|,\] which was subsequently improved by Conant \cite[Theorem 1.6]{gC21} to the current upper bound. Conant's methods do not use the full power of stability but instead work in the more general setting of finite VC-dimension. We later noticed that our (non-standard) techniques already implied the finer bound in terms of $H$. 

We also explore the interaction between model theory and recent work in
additive combinatorics in a second direction. In \cite{GS08} Green and Sanders
showed that subsets of a locally compact abelian group $G$ which are elements
of the Fourier algebra $\mathcal{A}(G)$ belong to the coset ring
$\mathcal{W}(G)$. They also gave an upper bound for the boolean complexity of
the representation as elements in
$\mathcal{W}(G)$ of such sets in terms of their Wiener norm. More recently,
Sanders \cite{tS18} showed that smallness of this norm implies stability, hence
$\mathcal{A}(G)=\mathcal{W}(G)\subseteq \mathcal{S}(G)$, where $\mathcal{S}(G)$
denotes the ring of stable subsets of $G$. He further observed that when $G$ is
not finite, it is possible for the latter inclusion to be strict. 

In this paper we shall consider the ring $\mathcal{WN}(G)$ of subsets
of $G$ generated by all instances of weakly normal formulae, defined in Section \ref{S:wn}. It is not difficult to see that
$\mathcal{WN}(G)$ is contained in yet not identical to the stability ring $\mathcal{S}(G)$. Actually, we have the chain of  inclusions \[ \mathcal{W}(G) \subseteq \mathcal{WN}(G) \subseteq \mathcal{S}(G).\]   
In fact, we shall show that
$\mathcal{WN}(G)$ is equal to $\mathcal{W}(G)$ for abelian $G$ (see Proposition 
\ref{P:wn}). This is a local reformulation of the celebrated
 result by Hrushovski and
Pillay \cite[Theorem 4.1]{HP87}. We further deduce a result of
Freiman-Ruzsa type for finite subsets in $\mathcal{WN}(G)$. 

Mimicking the
definition of Sanders in \cite{tS18}, we say that a subset $A$ of $G$ has an
\emph{$(r,k,l)$-weakly normal representation} if
\[A = \bigcup\limits_{j=1}^k B_j \cap \bigcup\limits_{i=1}^l G\setminus C_r, \]
where all the relations $B_1(x+ y),\ldots,B_k(x+ y),C_1(x+ y),\ldots,C_l(x+ y)$
are $r$-weakly normal.

\begin{teoB}\label{thm:teoB}
Given natural numbers $r$, $k$, and $l$,
there are natural numbers
$n=n(r,k,l)$ and $m=m(r,k,l)$ such that for any abelian
group $G$ and any subset $A\subseteq G$ with an $(r,k,l)$-weakly 
normal
representation, there are
subgroups $H_1,\ldots, H_n$
of $G$, each contained in $A-A$, with \[A\subseteq \bigcup\limits_{i=1}^{n}
g_i+H_i,\]
for some $g_1,\ldots, g_n$ in $A$. Furthermore, each $H_i$ is a boolean
combinations of complexity at most $m$ of translates of $A$.

In particular, if $A$ is finite, we have that $|A\cap
	(g+H)|\geq |A|/n$ for some $g$ in $A$ and some subgroup $H\leqslant G$ 
	contained in $A-A$.
\end{teoB}

In contrast to Theorem A, the subset $A$ above need not have small doubling or tripling. Indeed,  there is no correlation for finite sets between having a weakly normal
representation and small doubling: consider the group $G=\mathbb{F}_p^2$ and let $A$ be
the subset $(\mathbb{F}_p\times \{0\} \cup \{0\}\times\mathbb{F}_p)$, which has a
$(2,2,0)$-weakly normal representation. However, the quantity
\[ \frac{|A+A|}{|A|} = \frac{p^2}{2p-1}\]
is not uniformly bounded  for large $p$.

Throughout this paper, we will assume a certain familiarity with model theory. We refer the reader to
\cite{TZ12} for an excellent introduction to the subject. Basic notions related to local stability and weak normality are recalled and developed in Sections \ref{S:local} and \ref{S:wn}, respectively. Section \ref{S:measure} is devoted to a discussion of Keisler measures on a certain boolean algebra arising from sets of small tripling, and the associated measure-theoretic stabilizers. The proofs of our main results are given in Section \ref{S:main}.

\begin{ackno}
The authors wish to thank Gabriel Conant and Caroline Terry for many helpful
conversations and useful comments on a previous version of this note.
\end{ackno}

\section{Local stability}\label{S:local}

We work inside a sufficiently saturated model $\UU$ of a complete theory $T$
with infinite models in a language $\LL$.

Recall that a formula $\varphi(x,y)$ is \emph{$r$-stable} with respect to the
partition of the variables into the tuples $x$ and $y$ if there are no tuples
$a_1,\ldots, a_r, b_1, \ldots, b_r$ such that $\varphi(a_i,b_j)$
 holds if and only if  $i\leq j$.

A formula is stable if it is $r$-stable for some $r$. Stable formulae are
closed under boolean combinations (see \cite{TW18} for a finitary version of this fact). A set $X$ is \emph{$\varphi$-definable} over a
subset $A$ of parameters if it is
definable by a boolean combination of instances $\varphi(x,a)$ with $a$ in
$A$. By a \emph{$\varphi$-type} over a subset $A$ we mean a maximal finitely
consistent collection of instances of the form $\varphi(x,a)$ or
$\neg\varphi(x,a')$ for $a, a'$ in $A$.

The space of $\varphi$-types $S_\varphi(\UU)$ is a compact Hausdorff
$0$-dimensional
topological space, with basic clopen sets of the form
\[ [X] =\{ p\in S_\varphi(\UU)\,|\, p \cup\{X(x)\} \text{ is finitely
consistent} \},\]
where $X(x)$ is $\varphi$-definable.  Given a stable formula $\varphi(x,y)$ and a partial
$\LL$-type $\pi(x)$, the collection
\[X_\pi=\{q(x)\in  S_\varphi(\UU)\,|\, q(x) \cup \pi(x) \text{ is finitely
	consistent} \}\]
is a closed, hence compact,  subset of $S_\varphi(\UU)$ with integer-valued
\emph{Cantor-Bendixson rank} $\CAN_\varphi(\pi)$. Thus, any
element
$q(x)$ in $X_{\pi}$ can be isolated from all other types of rank at least $\CAN_\varphi(q)$ by the neighbourhood $[\chi]$ of some formula
$\chi(x)$. 
Furthermore,
the space $X_\pi$ contains only finitely
many elements
of maximal rank. The number of such elements is the
\emph{$\varphi$-multiplicity}
of $\pi$, see \cite[Chapter 6]{eC11}.

If $\varphi(x,y)$ is stable, then
every $\varphi$-type $p(x)$ over a small
submodel $M$ is \emph{definable}, that is, there is a formula $\theta(y)$ with
parameters over $M$ such that \[ \varphi(x,m) \in p \ \Longleftrightarrow \
\theta(m),\] for all $m$ in $M$.  Furthermore, the definable set $\theta(y)$
above is unique and can be defined by a positive boolean
combination of instances $\varphi(a,y)$ with parameters in $M$ (cf. \cite[Lemma
5.4]{HP94}). We refer to this definable set as the \emph{$\varphi$-definition}
$(d_p\varphi)(y)$ of $p$.  Given a superset $B\supseteq M$ of $\UU$, there is a
unique $\varphi$-type over $B$ extending $p$ which is again definable over $M$,
namely \[ \{\varphi(x,b) \,|\, (d_p\varphi)(b)\}\cup \{\neg\varphi(x,b') \,|\,
\neg(d_p\varphi)(b')\}.\]
We refer to this type as the \emph{non-forking extension} $p\nf B(x)$
of $p(x)$ to
$B$. The \emph{global non-forking extension} of $p$ is the $\varphi$-type $p\nf
\UU$. In fact,  the unique global non-forking extension of $p(x)$ is the only
element in $X_p$ of rank $\CAN_\varphi(p)$, so $p$ has $\varphi$-multiplicity
$1$ (cf. \cite[Proposition 6.13 \& Corollary 6.15]{eC11}).

\medskip
Henceforth, we will assume that the underlying structure $\UU$ carries a
definable group structure $(G,\cdot)$ without parameters. In order to
analyse the structure of an arbitrary stable subset $A$ of $G$, it suffices
expand the language by a distinguished unary predicate, whose realisations are
exactly
the elements in $A$. Thus, we may assume that the formula
$\varphi(x,y)= A(y\cdot x)$ is stable, for some fixed definable subset $A$ of
$G$.

Note that $\varphi(x,y)$ is \emph{equivariant} (see \cite[Definition
5.13]{HP94}), that is,
every
left-translate of an instance of $\varphi$ is again an instance of $\varphi$. 
Given a stable equivariant formula $\varphi(x,y)$, there is a distinguished 
subgroup of $G$ which is  $\varphi$-definable, relative to $G$, as  first 
observed in \cite{HP94}. The following fact can be found in 
\cite[Theorem 2.3]{CPT19}.
	
\begin{fact}\label{F:HrPi}
Given a stable equivariant formula $\varphi(x,y)$ and a definable group $G$ 
over a model $M$, there is a subgroup $G_\varphi^0$ of 
finite index in $G$ which is $\varphi$-definable over $M$ (relative to $G$) 
such that for any coset $C$ of $G_\varphi^0$ and any 
$\varphi$-definable subset $X$, either $X\cap C$ or $C\setminus X$ is 
generic, in the sense that finitely many translates cover $G$.
\end{fact}	

The Cantor-Bendixson rank of a union is the maximum of 
the ranks of 
the sets in the union, so every generic $\varphi$-definable subset 
of $G$ has maximal Cantor-Bendixson 
rank $\CAN_\varphi(G(x))$. On the other hand, if $X$ is a 
$\varphi$-definable subset of $G$ of maximal 
Cantor-Bendixson rank $\CAN_\varphi(G(x))$, it must be generic: 
indeed, since $G_\varphi^0$ has finite index in $G$, there must 
be a coset $C$ of $G_\varphi^0$ 
such that $X\cap C$ has rank $\CAN_\varphi(G(x))$. We need only 
show that $C\cap X$ is generic. Otherwise, the set $C\setminus 
X$ is generic, so finitely many translates will cover $G$. Each such 
translate has maximal rank, yet every coset of $G^0_\varphi$ 
contains a unique $\varphi$-type of maximal rank.

Given a $\varphi$-type $p(x)$ over a submodel $M$, we define its 
\emph{stabilizer} to
be
the subgroup \[\stab_\varphi(p)= \big\{g\in G\,|\, \forall u \big(
(d_p\varphi)(u)
\leftrightarrow (d_p\varphi)(u\cdot g) \big)   \big\}. \]
The stabilizer is clearly a definable subgroup of $G$ with parameters from $M$. 
 The 
following elementary
remark shows that the stabilizer is $\varphi$-definable whenever $G$ is abelian.

\begin{remark}\label{R:stab_def}
If $(G,+)$ is abelian, then the subgroup $\stab_\varphi(p)$ is
$\varphi$-definable
over $M$.
\end{remark}
\begin{proof}
Let $q(x)$ be the unique global non-forking extension of $p(x)$. Choose a
$\varphi$-formula \[ \chi(x)= \bigvee\limits_{j\in J}
\bigwedge\limits_{i\in I}  \varphi(x,b_{ij}) \land\neg\varphi(x,c_{ij}) \]
such that $q$ lies in the neighborhood $[\chi]$, with $\chi(x)$ of
Cantor-Bendixson rank $\CAN_\varphi(p)$ and $\varphi$-multiplicity $1$.

Now, an element $g$ in $G$ belongs to $\stab_\varphi(p)$ if and only if the
$\varphi$-type $g+q$ equals $q$, that is, if and only if $
\chi(x)-g$ belongs to $q$, or equivalently, if and only if

\[ \bigvee\limits_{j\in J}
\bigwedge\limits_{i\in I}  (d_p\varphi)( b_{ij} +g)
\land\neg(d_p\varphi)( c_{ij}+g). \]
Recall that $(d_p\varphi)(y)$ is a positive boolean combination of instances
$\varphi(a, y)$. Since $G$ is abelian, the formula $\varphi(x,y+z)$ is
equivalent to $\varphi(x+ y, z)$, so the above condition on $g$ is
equivalent to a boolean combination $\psi(z,a')$ of instances of $\varphi(a',
z)$,
for some choice of parameters $a'$ in $G$.  In particular, the
formula
\[ \exists u \forall z\big(\stab_\varphi(p)(z) \leftrightarrow \psi(z,u)\big)
\]
 holds in $\UU$. Since $M$ is an elementary substructure of $\UU$, there are
 some parameters $m$ in $M$ such that $\stab_\varphi(p)(M)$ equals $\psi(M,
 m)$, and thus the $\varphi$-formula $\psi(z, m)$ defines the subgroup
 $\stab_\varphi(p)$ in $G$.
\end{proof}

Note that if $G$ is abelian, then $\varphi(x,y)=\varphi(y,x)$. Given global $\varphi$-types $p(x)$ and
$q(y)$ in $S_\varphi(\UU)$, Harrington's lemma \cite[Lemma 6.8]{eC11}
yields
that \[  q(y) \in [(d_p\varphi)(y)] \Leftrightarrow p(x) \in
[(d_q\varphi)(x)].\]
A standard argument yields the following result, whose short
proof we include for completeness.

\begin{remark}\label{R:rank_Stab}
If $(G,+)$ is abelian, then given a $\varphi$-type $p$ over $M$  \[
\CAN_\varphi(\stab_\varphi(p)) \leq \CAN_\varphi(p). \]
\end{remark}
\begin{proof}
Let $q$ be a global type in $[\stab_\varphi(p)]$ of maximal rank, and choose a
realization  $b$ of $q\restr M$. Note that $q$ is a non-forking extension
of $q\restr M$, since $\stab_\varphi(p)$ is definable over the model $M$. Let
$a$ realize the non-forking extension
$p\nf
{M\cup\{b\}}$, which is definable over $M$ by the formula $(d_p\varphi)(y)$.
Thus, the element $a+b$ realizes $p$, since $-b$ belongs to
$\stab_\varphi(p)$.

Let us first show that $b$ realizes the non-forking extension $q\restr
{M\cup\{a\}}$ of $q\restr M$. It suffices to see that $\varphi(a,b)$ holds if
and only if $(d_q\varphi)(a)$. Now,
\begin{align*}
(d_q\varphi)(a)  & \Longleftrightarrow p\nf \UU(x) \in
[(d_q\varphi)(x)]   \stackrel{\text{Harrington}}{\Longleftrightarrow}  q(y) \in
[(d_p\varphi)(y)] \\ &
\Longleftrightarrow (d_p\varphi)(b) \Longleftrightarrow
\varphi(x,b)  \in
p\nf {M\cup\{b\}}   \Longleftrightarrow \varphi(a,b) \text{ holds.}
\end{align*}
As the formula $\varphi$ is equivariant, addition by an element preserves the
rank of formulae, so
\begin{align*}
\CAN_\varphi(\stab_\varphi(p)) & = \CAN_\varphi(q\restr M) =
\CAN_\varphi(q\restr {M\cup \{a\}}) = \CAN_\varphi(b/M\cup \{a\}) \\
	&= \CAN_\varphi(a+b/M\cup \{a\})\leq \CAN_\varphi(a+b/M)=\CAN_\varphi(p),
\end{align*}
as desired.
\end{proof}

\section{Weak Normality}\label{S:wn}
Given a natural number $k$, a formula
$\psi(x,y)$ is \emph{$k$-weakly normal} if, whenever the
instances $\psi(x,b_1), \ldots, \psi(x,b_k)$ are pairwise distinct,
the intersection $\bigcap_{i=1}^k \psi(x,b_i)$ is empty \cite{HP87}.  A formula is weakly
normal if it is $k$-weakly normal for some natural number $k$. The conjunction
of weakly normal formulae is again weakly normal. However, neither the negation
nor the disjunction of two weakly normal formulae need necessarily be weakly normal.

It is easy to see that a $k$-weakly normal formula is $k$-stable. If not,
there is a sequence $(a_i, b_i)_{1\leq i\leq k}$ witnessing the failure of
stability. Since $a_j$ belongs to $\psi(x,b_j)$ but not to $\psi(x,b_i)$ for
$i<j$, the instances are pairwise distinct. However, the element $a_1$ belongs
to their common intersection, so $\psi(x,y)$ is not
$k$-weakly normal. Furthermore,  a formula $\psi(x,y)$ is  $2$-stable precisely
if it is
$2$-weakly normal.  Indeed, if $\psi(x,y)$ is not $2$-weakly normal,
we can find  two distinct instances $\psi(x,b_1)$ and $\psi(x,b_2)$ with non-empty
intersection. We may
assume that there is some $a_2$ in $ \psi(x,b_2)\setminus \psi(x,b_1)$.
As the intersection $\psi(x,b_1)\cap \psi(x,b_2)$ is non-empty,  choose $a_1$
in  $\psi(x,b_1)\cap\psi(x,b_2)$ and note that \[ \psi(a_i,b_j) \Leftrightarrow
1\leq i\leq
j\leq 2,\]
so $\psi(x,y)$ is not $2$-stable.

Formulae which are 2-stable are very special. For example, in the setting of a
group $G$ with a fixed definable subset $A$, the formula $\varphi(x,y)=A(y\cdot
x)$ is $2$-stable if and only if $A$ is either empty or a coset of a subgroup
of $G$. Recall that a subset
$A$ of an abelian group $G$ is Sidon if, whenever the $4$-tuple 
$(a_1,a_2,a_3,a_4)$ of elements
of $A$ satisfies
$a_1-a_2=a_3-a_4$, then $a_1=a_2$ (and hence $a_3=a_4$) or $a_1=a_3$ (and thus
$a_2=a_4$). Sidon subsets of the integers, such as $2^{\mathbb N}$ or
$3^{\mathbb N}$, are 3-stable, but need not lie in the coset ring
$\mathcal{W}(\mathbb Z)$ \cite{tS18}.

\begin{remark}\label{R:stab_wn}
In general, stability need not imply weak normality. For a Sidon set $A$ of 
cardinality at least $k$, the formula  $A(x+y)$ cannot be $k$-weakly normal. 
 Choose $k$ distinct elements $a_1,\ldots, a_k$ in $A$ and consider the 
 collection of sets $(-a_j+A)_{1\leq j\leq k}$.  The element $0$
 belongs to their
 common intersection, yet they are pairwise distinct sets.
 
\end{remark}

Since a definable set is defined over a submodel $N$ if and only if it only has
finitely many distinct automorphic copies over $N$ (see for example \cite[Proposition
1.11]{eC11}), we deduce the following
easy observation concerning sets defined by an instance of a weakly normal formula.

\begin{remark}\label{R:defset_wn}
Let $X$ be a definable set given by an instance of a weakly normal formula. Then the
set $X$ is definable over any submodel containing a realization of $X$.
\end{remark}

A remarkable property of every weakly normal formula $\psi$ is that the definition
$(d_p\psi)$ of every local type $p$
over an arbitrary set of parameters is explicit, in contrast to a general
stable formula (cf. \cite[Theorem 8.3.1]{TZ12}): indeed, given a $k$-weakly
normal formula $\psi(x,y)$, an instance $\psi(x,a)$ belongs to the $\psi$-type
$p=\tp_\psi(c/A)$ if and only if it contains the set \[ \mathcal{X}_{p,\psi}=
\bigcap\limits_{\psi(x,a') \in p} \psi(x,a').\]
This set is definable since it is the intersection of at most $k-1$
instances in $p$ (notice that $\mathcal{X}_{p,\psi}$ is the empty set 
if and only if  the collection of positive instances $\psi(x,a')$  in $p$ 
is empty). It suffices to
set \[(d_p\psi)(y) =\forall x \big( \mathcal{X}_{p,\psi}(x) \rightarrow
\psi(x,y) \big) \text{ when $\mathcal{X}_{p,\psi}\neq \emptyset$}, \] and
\[(d_p\psi)(y) = (y\neq y) \text{ otherwise.}\]

\begin{remark}\label{R:deftype_wn}
Assume that the formula $\chi(x,y)$ is a boolean combination of weakly normal
formulae. Then every $\chi$-type $p$ is definable over any
submodel containing a realization of $p$.
\end{remark}
Note that we do not require that the submodel contains the parameter set of
$p$, which is assumed to be small with respect to the saturation of $\UU$. 

\begin{proof}
If the formula $\chi(x,y)$ is a boolean combination of the weakly normal
formulae $\psi_1(x,y),\ldots, \psi_r(x,y)$, then the $\chi$-type $p=\tp_\chi(c/A)$
is
determined by the collection of types $q_1=\tp_{\psi_1}(c/A), \ldots,
q_r=\tp_{\psi_r}(c/A)$. Hence, the $\chi$-definition $(d_p\chi)$ is determined
by the definable sets $\{(d_{q_i}\psi_i)\}_{1\leq i\leq r}$. Each
$(d_{q_i}\psi_i)$ is determined by the corresponding definable set
$\mathcal{X}_{q_i,\psi_i}$, as in the previous discussion, which is definable
over any submodel containing $c$, by Remark \ref{R:defset_wn}.
\end{proof}

A well-known result of Hrushovski and Pillay \cite[Lemma 4.2]{HP87} shows that, 
in a theory where all formulae are boolean combinations of weakly normal ones,
types are generic in cosets of their stabilizers. In particular, groups
definable in such theo\-ries are virtually abelian, that is, abelian-by-finite.
We will provide a local version of their results for abelian groups, following
closely \cite[Lemma 2.6 \& Remark 2.7]{aP02}.
\begin{lemma}\label{L:wn_stab}
Let $(G,+)$ be abelian and assume that the formula $\varphi(x,y)=A(x+y)$ is a
boolean combination of weakly normal formulae. Given a $\varphi$-type $p$ over
a model $M$, there exists  some element $m$ in $M$ such that $p\nf \UU$ lies in
the neighborhood $[m+\stab_\varphi(p)]$, that is, the type $p$
implies the $\varphi$-formula over $M$ defining the coset $m+\stab_\varphi(p)$.
\end{lemma}
Furthermore, the proof of the above result yields that 
$\CAN_\varphi(\stab_\varphi(p)) = \CAN_\varphi(p)$, but this fact will not be used 
in the sequel.
\begin{proof}
Let $a$ be a realization of $p(x)$. Since $G$ is abelian, every coset of
$H=\stab_\varphi(p)$ is $\varphi$-definable (since $\varphi(x,y)$ is 
equivariant). We want to show that the coset
$H+a$ is $\varphi$-definable over $M$. It suffices to show that it is definable
over $M$, for $M$ is an elementary substructure. Since the element $a$ lies in
$H+a$, if this coset is definable over $M$, then the type $p$ must imply the
corresponding $\varphi$-formula over $M$ defining it. Remark
\ref{R:rank_Stab} yields then the equality of ranks.

To prove that $H+a$ is definable over $M$, we need only show that $H+a$  is
definable over a submodel $N\succeq M$ such that $\tp(a/N)$ is an heir of
$\tp(a/M)$: indeed, suppose that $H+a$ is definable over such $N$, so there are an
$\LL_M$-formula $\theta(x,z)$ and some tuple $n$ in $N$ such that the formula
$\theta(x,n)$ defines $H+a$. Note that this coset is definable over
$M\cup\{a\}$. In particular, the formula
\[ \forall x \big((H+u)(x)\leftrightarrow \theta(x,n) \big)\]
belongs to $\tp(a/N)$. Thus, we
find a tuple $m$ in $M$ such that $\theta(x, m)$ defines $H+a$, by inheritance
of $\tp(a/N)$ over $M$.

Now choose a $\varphi$-type $q(x)$ over $M$ of maximal Cantor-Bendixson rank
$\CAN_\varphi(G)$. The extension $q\nf {M\cup\{a\}}$ is definable over $M$, so
it is finitely satisfiable over $M$ \cite[Lemma I.2.16]{aP96}. Thus we can find
some element $c$ such that the complete type $\tp(c/M\cup \{a\})$ is finitely
satisfiable over $M$ and extends $q\nf {M\cup\{a\}}$.  By a dual argument, we
can find a submodel $N\succeq M$ containing $c$ such that $\tp(a/N)$ is an heir
of $\tp(a/M)$.

\begin{claimstar}
The element $a$ realizes $p\nf {M\cup\{c-a\}}$.
\end{claimstar}
\begin{claimstarpf}
We need only show that $\varphi(a,c-a)$ holds if and only
if $(d_p\varphi)(c-a)$ holds. Observe first that
\begin{align*}
\CAN_\varphi(G) & = \CAN_\varphi(q) =  \CAN_\varphi( q\nf {M\cup \{a\}}) =
\CAN_\varphi(\tp_{\varphi}(c/M\cup \{a\})) \\
&= \CAN_\varphi(\tp_{\varphi}(c-a/M\cup \{a\}))\leq
\CAN_\varphi(\tp_{\varphi}(c-a/M))\leq \CAN_\varphi(G).
\end{align*}
Hence, equality holds everywhere, so $r(y)=\tp_{\varphi}(c-a/M\cup \{a\})$ is
definable over $M$ and has maximal rank $\CAN_\varphi(G)$.

Now, the formula $\varphi(a,c-a)$ holds if and only if $\varphi(a,y)$ belongs
to $r(y)$, that is, if and only if the element $a$ realizes
$(d_r\varphi)(x)$, which is definable over $M$. Hence, the formula
$\varphi(a,c-a)$ holds if and only if $p\nf \UU$ lies in $[(d_r\varphi)]$,
which is equivalent to $r\nf \UU$ being contained in $[(d_p\varphi)]$, by
Harrington's lemma. Since $[(d_p\varphi)]$ is definable over $M$ and the
element $c-a$ realizes $r$, the latter is equivalent to $(d_p\varphi)(c-a)$, as
desired.
\end{claimstarpf}

Since $c=a+(c-a)$, the element $c$ realizes the complete $\varphi$-type \[ p\nf
{M\cup\{c-a\}} +(c-a)=\{ \theta(x) \ \varphi\text{-formula over }M\cup\{c-a\} \ | \ \theta\left(x-(c-a)\right) \in p\nf
{M\cup\{c-a\}} \},\] which is again a complete $\varphi$-type over $M\cup\{c-a\}$. In particular, the global
$\varphi$-type $p\nf \UU +(c-a)$ is a non-forking extension of $p\nf
{M\cup\{c-a\}} +(c-a)$. By Remark \ref{R:deftype_wn}, both types are definable
over $N$ (which contains $c$).

Let us now show that the coset $H+a$ is definable over $N$. It suffices to show
that every automorphism $\sigma$ fixing $N$ pointwise fixes the coset
(setwise). Since $p\nf \UU +(c-a)$ is definable over $N$, the automorphism
$\sigma$ fixes $p\nf \UU +(c-a)$, so
\[ p\nf \UU +(c-a) = \sigma(p\nf \UU +(c-a)) = p\nf \UU +(c-\sigma(a)), \]
since $p\nf \UU$ is definable over $M$. Thus, we have $p\nf \UU -a = p \nf \UU
-\sigma(a)$, that is,
 \[ p\nf \UU  = p \nf \UU  + (a-\sigma(a)), \]
and hence $a-\sigma(a)$ lies in $H=\stab_\varphi(p)$, as desired.
\end{proof}

In analogy to the classical result for weakly normal theories, we conclude that
$\varphi$-definable sets are boolean combination of cosets of $\varphi$-definable
groups whenever $\varphi(x,y)=A(x+y)$ is a boolean combination of weakly normal
formulae.

\begin{cor}\label{C:BC}
In an abelian group $G$ (written additively), assume that the formula 
$\varphi(x,y)=A(x+y)$ is a
boolean combination of weakly normal formulae. Every $\varphi$-definable set is a
boolean combination of cosets of $\varphi$-definable subgroups.
\end{cor}
Note in particular that a coset of a $\varphi$-definable subgroup is a boolean combination of
translates of $A$.

\begin{proof}
By a straightforward application of \cite[Lemma 3.1.1]{TZ12}, it suffices to show that
whenever two $\varphi$-types $p_1$ and $p_2$ over a submodel $M$ imply the same
$M$-definable
cosets of $\varphi$-definable subgroups (over $M$), then $p_1$ and $p_2$ are the same.

Let $a_1$ realize $p_1$ and choose a realization $a_2$ of ${p_2} \nf {M\cup\{a_1\}}$. Set
$H_1=\stab_\varphi(p_1)$ and $H_2=\stab_\varphi(p_2)$.  By Lemma \ref{L:wn_stab}, both
cosets $a_1+H_1$ and $a_2+H_2$ are $M$-definable. By assumption, since $p_1$ 
clearly implies the formula defining $a_1+H_1$, every realization of $p_2$ lies 
in $a_1+H_1$, and similarly for $p_1$. In particular, the element
$a_1-a_2$ lies in $H_1\cap H_2$.  The rank computation
\begin{align*}
\CAN_\varphi(H_2) & \le \CAN_\varphi(p_2)=\CAN_\varphi({p_2} \nf {M\cup\{a_1\}}) =
\CAN_\varphi(\tp_{\varphi}(a_2/M\cup \{a_1\}))\\
&= \CAN_\varphi(\tp_{\varphi}(a_2-a_1/M\cup \{a_1\}))\leq
\CAN_\varphi(\tp_{\varphi}(a_2-a_1/M))\le \CAN_\varphi(H_2)
\end{align*}
yields that $a_2-a_1$ realizes the non-forking extension of $q=\tp_{\varphi}(a_2-a_1/M)$
over $M\cup\{a_1\}$. As in the proof of Remark \ref{R:rank_Stab}, Harrington's Lemma implies that $a_1$
realizes ${p_1} \nf
{M\cup\{a_2-a_1\}}$. Since $a_2-a_1$ lies in $H_1=\stab_\varphi(p_1)$, we have that
$a_2=a_1+(a_2-a_1)$ realizes $p_1$. Thus the types $p_1$ and $p_2$ are equal, as
desired.
\end{proof}

\section{Ideals and measures}\label{S:measure}

A \emph{Keisler measure} $\mu$ is a finitely additive probability measure on some boolean
algebra of definable subsets of the ambient model \cite{jK87}. Archetypal examples are
measures $\mu_p$, with two possible values $0$ and $1$, given by global
$\varphi$-types
$p$, that is, for every $\varphi$-definable set $X$,
\[ \mu_p(X)=1 \Leftrightarrow p \in [X].\]

\noindent Given a Keisler measure $\mu$, the collection of sets of
measure zero forms an {\em ideal}, that is, it is closed under
subsets and finite unions. A partial type is said to be \emph{wide} (with respect to $\mu$)
if it contains no definable set of measure zero. In particular, since the collection of
measure-$0$ sets forms an ideal, every wide partial type $\pi(x)$ over a parameter set $A$
can be completed to a wide complete type over any arbitrary subset $B$ containing $A$:
indeed, the collection of formulae
\[ \pi(x)\cup\{\neg\varphi(x) \,|\, \varphi(x)\in \mathcal L_B \text{ with } \mu(\varphi(x))=0 \}\]
is clearly finitely consistent, and any completion of this partial type is wide. 
Note that we do not require that every formula in the completion has 
measure.

The measure $\mu$ is said to be \emph{definable over the submodel} $M$ 
(see \cite[Definition 3.19]{sS17}) if for every
$\LL$-formula $\varphi(x,y)$ and every $\epsilon>0$, there is a  partition of
$\UU^{|y|}$ into $\LL_M$-formulae $\rho_1(y),\ldots, \rho_m(y)$ such that for all pairs
$(b, b')$  realizing $\rho_i(y)\land\rho_i(z)$, we have that
\[ |\mu(\varphi(x,b)) -\mu(\varphi(x,b') )|<\epsilon. \]
\noindent In particular, the set of tuples $b$ with $\mu(\varphi(x,b))=0$ is type-definable
over $M$
 and the map
 \[\begin{array}{ccc}
  S_y(M) & \to &  [0,1]\\[1mm]
 \tp(b/M) & \mapsto & \mu(\varphi(x,b))

 \end{array}  \] is continuous, so the value $\mu(\varphi(x,b))$ only depends on $\tp(b/M)$. Note that a global
 $\varphi$-type $p$ is definable over $M$ if and only if the corresponding measure
 $\mu_p$  is.

 Every Keisler measure admits an expansion of the original
 language $\LL$ in which it becomes definable (cf. \cite[Section
 2.6]{eH12}). In this case, a formula of positive measure does not fork over $\emptyset$,
 see \cite[Lemma 2.9 \& Example 2.12]{eH12}.

 In the presence of an ambient group $G$, we will consider the following notion of an \emph{acceptable} set, which
 was introduced as a \emph{near-subgroup} in \cite[Definition 3.9]{eH12}.

 \begin{definition}\label{D:acceptable}
  A definable subset $A$ of $G$ is \emph{acceptable} if there exists a Keisler
  measure  $\mu$ on a boolean algebra of definable subsets of $(A\cup
  A^{\inv}\cup\{\mathrm{id}_G\})^3$ such
  that $\mu(A)>0$, the set $A\cup
  A^{\inv}\cup\{\mathrm{id}_G\}$ has measure $1$ and $\mu(Y)=\mu(X)$ for all definable measurable subsets
  $X$   and $Y$   of  $(A\cup A^{\inv} \cup\{\mathrm{id}_G\})^3$ whenever $Y$ is a translate of $X$.
 \end{definition}

\begin{example}\label{E:amenable}
	Let $G$ be an abelian group, or more generally, an amenable group, equipped with a
	finitely additive probability measure $\mu$. Every subset of positive measure becomes
	an acceptable
	subset of $G$ witnessed by the restriction of $\mu$ with respect to a suitable
	boolean algebra of $(A\cup A\inv \cup\{\mathrm{id}_G\})^3$.  As above, we can
	expand the language of groups to a suitable language $\mathcal L$ in such a way that
	both $A$ and the measure $\mu$ are definable.
\end{example}

\begin{example}\label{E:tripling}
	Consider a finite non-empty subset $A$ of a (possibly infinite group) $G$ with
	tripling $K$, that is, with $|A\cdot A\cdot A| \leq K|A|$. Then $B=A\cup A\inv\cup
	\{\mathrm{id}_G\}$ has size at least $|A|$ and most $2 |A|+1$, and it follows from Ruzsa calculus that $|B\cdot B\cdot B|\leq 14K^3 |A|$. Given a subset
	$X \subseteq B\cdot B\cdot B$,
	set \[ \mu(X) = \frac{|X|}{ |B\cdot B\cdot B|}.\]
	We have thus obtained a finitely additive measure $\mu$ 
	such that $\mu(A)\geq (14K^3)^{-1}$. Hence, the set $A$ is acceptable (with respect to
	the measure $\mu$).

	Furthermore, if $G$ is abelian, it follows from the Pl\"unnecke-Ruzsa inequality \cite[Corollary 6.29]{TV06} that we need only assume that $A$ has small doubling.
\end{example}

Given an acceptable subset $A$ of $G$ with respect to the finitely additive 
definable measure $\mu$ and a complete type $p(x)$ over a
submodel $M$ containing
the formula $A(x)$, define its \emph{(measure-theoretic) stabilizer}
$\stab(p)$ to be
the group generated by the set
$$
\mathrm{st}(p)=\{g\in G: g\cdot p(x) \cup p(x) \text{ is wide}\}.
$$
Note that $\mathrm{st}(p)$ contains the identity element of $G$, whenever the type $p$ is wide. The set
$\mathrm{st}(p)$, and hence $\stab(p)$, is invariant under automorphisms of $\UU$
fixing $M$ pointwise.

Inspired by the corresponding results in geometric stability theory, Hrushovski proved in
\cite[Theorem 3.5]{eH12} that the measure-theoretic stabilizer  is type-definable and
equals the set
$\mathrm{st}(p)\cdot \mathrm{st}(p)$, whenever $p$ is a wide type containing an
acceptable subset $A(x)$
with respect to some (definable) measure $\mu$.  Furthermore,  the
stabilizer is a normal
subgroup of the group $\langle A \rangle$ generated by $A$ and has of bounded index in $\langle A \rangle$, that is, 
the number of cosets of $\stab(p)$ in $\langle A \rangle$ is bounded by the cardinality of saturation
of the ambient model $\UU$. For our purposes, we need a much
weaker statement, namely, that $\mathrm{st}(p)$ contains some wide complete type, for
which we will now give a simple proof.

\begin{lemma}\label{L:st_wide}
 	Let $A$ be an acceptable subset of $G$ with respect to the measure $\mu$, which we
 	assume to be definable over a submodel $M$. Given a wide type $p(x)$ over  $M$
 	containing
 	the formula $A(x)$, there exists a wide type $q(x)$ over $M$ whose realizations belong to the $M$-invariant set $\mathrm{st}(p)$.
 \end{lemma}

 \begin{proof}
 	Consider a sequence $(a_i)_{i\in\N}$ of realizations of $p$ such that $\tp(a_i/M\cup
 	\{a_j\}_{j<i})$  is wide. By a standard Ramsey argument, we may assume that the
 	sequence
 	is indiscernible over $M$. Set $q=\tp(a_2\inv\cdot a_1/M)$. Since $\mathrm{st}(p)$ is $M$-invariant, it suffices to show that the realization 
 	$a_2\inv\cdot a_1$ of $q$ belongs to $\mathrm{st}(p)$. Otherwise, the 
 	partial type
 	$a_2\inv \cdot a_1\cdot p(x)\cup p(x)$ is not wide and neither is $a_1\cdot p(x)\cup a_2
 	\cdot p(x)$ by translation-invariance of the measure (for $A$ is acceptable 
 	and $a_2$ is an element of	$A$). Thus,
 	we
 	can find some $M$-definable set $X(x)$
 	in $p(x)$ contained in $A(x)$ such that \[ \mu( a_j\cdot X \cap a_i\cdot X)= 0
 	\text{ for } i\neq j.\]
 	As the definable set $X$ is wide, there exists some natural number $k$ such that
 	$\frac{1}{k}< \mu(X)=\mu(a_j\cdot X)$ for all $j$ in $\N$, so
 	\[\mu\Big(\bigcup_{i=1}^k a_i \cdot X\Big) = \sum_{i=1}^k \mu(a_i \cdot X) = k \cdot
 	\mu(X)>1,\]
 	which contradicts the assumption that $\mu$ is a probability measure.
 \end{proof}

 A standard application of Ruzsa's covering argument (cf. \cite[Lemma 2.14]{TV06}) yields the following auxiliary result, which resonates with Lemma \ref{L:wn_stab}.

 \begin{lemma}\label{L:wide_fteindex}
 	Let $A$ be an acceptable subset of $G$ with respect to the measure $\mu$, which we
 	assume to be definable over a submodel $M$. Given a wide
 	$M$-definable subgroup $H\subseteq A\cdot
 	A\inv$ of $G$, there is some finite subset $C$ of $A(M)$ such that $A$ is contained in $C\cdot H$.
 \end{lemma}
 \begin{proof}
 	Note that $a\cdot H\subseteq (A\cup A\inv)^3$ is wide for every $a$ in $A$. Hence,
 	choose a maximal subset $C$ of $A$ (in $\UU$) such that $(c\cdot H)  \cap (c'\cdot H)
 	=\emptyset$
 	for every $c\neq c'$ in $C$. In particular, given any $a$ in $A$, there is some $c$ in
 	$C$
 	such that $(a\cdot H)  \cap (c\cdot H)$  is non-empty.
 	Thus, the element $a$ lies in $c\cdot H$ and so $A \subseteq C\cdot H$.

 	For each $c$ in $C$, we have that $\mu(c\cdot H)=\mu(H)>\frac{1}{k}$ for some $k$ in
 	$\N$. As in the proof of Lemma \ref{L:st_wide},  we deduce that $C$ is finite.
 	Since both $A$ and $H$ are definable over the model $M$ and $A$ is contained in
 	$C\cdot H$, we can take $C$ to be a subset of $A(M)$, as
 	desired.
 \end{proof}

\begin{prop}\label{P:stable_acceptable}
 	Let $A$ be  an acceptable subset of $G$ with respect to the measure $\mu$,
 	which we
 	assume to be definable over a submodel $M$, and assume that the formula $\varphi(x,y)=A(y\cdot
 	x)$ is stable.  Then there exists some $M$-definable subgroup $H$ of $G$ contained in $A\cdot A\inv$ such that $A\subseteq C\cdot H$ for some finite subset $C$ of $A(M)$.

 	Furthermore, if $G$ is abelian, then $H$ can be taken to be a boolean
 	combination of
 	translates of $A$.
 \end{prop}

 \begin{proof}
 Since  the  definable set $A(x)$ is wide, we may extend it to a wide complete
 $\LL$-type
 $p(x)$ over $M$.  The $\varphi$-type
 $q=p\restr\varphi$ contains the instance $\varphi(x,\mathrm{id}_G)=A(x)$, so the $M$-definable
 group $H=\stab_\varphi(q)$ is clearly contained in $A\cdot A\inv$.

In order to conclude the result by Lemma \ref{L:wide_fteindex} (together with Remark \ref{R:stab_def} when $G$ is abelian), it suffices to show that $H$ is wide.
  By Lemma
 \ref{L:st_wide}, the set \[\mathrm{st}(p)=\{g\in G(\UU): g\cdot p(x)
 \cup p(x) \text{ is wide}\} \]  contains a wide type over $M$. So  we need only show that
 $\mathrm{st}(p)\subseteq H$.  Let $g \in \mathrm{st}(p)$ and denote by $q\nf 
 \UU (x)$ in $S_\varphi(\UU)$
 	the global non-forking extension of $q=p\restr\varphi$. Since formulae of positive
 	measure do not fork over $\emptyset$, we have that the partial type $g\cdot q(x)
 	\cup q(x)$ does not fork over $M$, hence it is a restriction of $q\nf 
 	\UU (x)$. In
 	particular, the global $\varphi$-type $g\inv\cdot q\nf 
 	\UU (x)$ is a non-forking extension of
 	$q(x)$.
 	By uniqueness of the global non-forking extension, we conclude that 
 	\[g\inv\cdot q\nf  	\UU (x)= q\nf  	\UU (x),\]
 	so $g\inv$, and thus $g$, lies in $H$ as desired.
  \end{proof}

\section{Main results}\label{S:main}

We are now in a position to prove our main results. Let us begin by recalling Theorem A in the introduction for the reader's convenience.

\begin{theorem}\label{T:smalltripling}
Given  real numbers $K\geq 1$ and $\epsilon>0$ and a natural number $r\geq 
2$, 
there
exists a natural number $n=n(K, \epsilon, r)$ such that
for any (possibly infinite) group $G$ and any finite $r$-stable subset 
$A\subseteq G$ with tripling $K$, there is a
subgroup $H\subseteq A\cdot A\inv$
of $G$ with $A\subseteq C\cdot H$ for some $C\subseteq A$ of size at most $n$. Moreover, there exists $C'\subseteq C$ such that 
\[ |A \triangle (C'\cdot H)|<\epsilon |H|.\]
\end{theorem}
\noindent In particular, we conclude that 
\[ |A \triangle (C'\cdot H)|<(\epsilon K^2) |A|,\] by the Pl\"unnecke-Ruzsa inequality.
\begin{remark} It follows from Proposition
\ref{P:stable_acceptable} that when $G$ is abelian, the subgroup $H$ can be taken to be a boolean
combination (whose complexity only depends on $K$, $\epsilon$ and $r$) of
translates of $A$.
\end{remark}

\begin{proof}
The proof proceeds by contradiction. Assuming that the statement does not hold, 
there are fixed $K$, $\epsilon$ and $r$ such that for each $n$ in $\N$, we 
find a finite
$r$-stable
subset $A_{n}$
of a group $G_{n}$ with tripling $K$ such that there are no subgroup 
$H\subseteq A_{n}\cdot A_{n}\inv$ and a finite subset $C$ of $A_n$ of size $n$ 
with $A_{n}$ contained in 
$C\cdot H$ and $|A_n \triangle (C'\cdot H)|< \epsilon |H|$ for 
some  $C'\subseteq C$.

Following the approach of \cite[Section 2.6]{eH12} (see also \cite[Proof 
of Theorem 1.3]{CPT19} and \cite[Section 2.3]{dP19}), we consider a suitable
expansion $\LL$ of the language of groups and regard each group $G_{n}$ as an
$\LL$-structure $M_{n}$. Choose a non-principal ultrafilter $\mathcal U$ on
$\N$
and consider the
ultraproduct $M=\prod_{\mathcal U} M_{n}$. The language $\LL$ is chosen in
such a
way that the sets $A=\prod_{\mathcal
	U}  	A_{n}$ and
	$B=\prod_{\mathcal U}  B_{n}$ are $\LL$-definable in the group
	$G=\prod_{\mathcal U} G_{n}$, where $B_{n}=A_{n}\cup
	A_{n}\inv\cup\{\mathrm{id}_G\}$. Furthermore, the
	counting measure \[ \mu_n(X_n) = \frac{|X_n|}{ |B_{n}\cdot
	B_{n}\cdot B_{n}|},\] induces a definable Keisler measure $\mu$, namely
	the standard
	part of $\lim_{\mathcal U} \mu_{n}$,
on the boolean algebra of $\LL$-definable subsets of $B\cdot B\cdot B$, such
that the set $A$
is acceptable. Now choose a sufficiently saturated elementary extension $\UU$ of the model
$M$ and note that $\mu$ is definable over $M$. Also,
every collection of subgroups of $M_n$ induces an $M$-definable group in
$G(\UU)$, and vice versa.

Note that the formula $\varphi(x,y)=A(y\cdot x)$ is $r$-stable, by \L o\'s's
theorem. Moreover, by construction the set $A$ is acceptable. Hence,
by Proposition \ref{P:stable_acceptable}, there is an
$M$-definable subgroup $H$ contained in $A\cdot A\inv$ such that $A\subseteq 
C\cdot H$, for some finite set $C$ in $A(M)$. By Fact \ref{F:HrPi}, after 
possibly increasing the size of $C$, we may assume that $H$ 
is $H_\varphi^0$. 

Let $C'$ be the collection of coset representatives $c$ in $C$ such that 
$A\cap 
(c\cdot H)$ is wide. Since $\mu$ is finitely additive and $A\subseteq C\cdot H$, 
we have that

\[ \mu(A)= \sum\limits_{c\in C} \mu\big(A \cap  (c\cdot H)\big) = 
\sum\limits_{c\in C'} 
\mu\big(A \cap  (c\cdot H)\big) = \mu\big(A \cap (C'\cdot H)\big),\]
so \[ \mu\big(A \setminus (C'\cdot H)\big) = \mu(A) - \mu\big(A \cap (C'\cdot 
H)\big) =0.\]
Thus, in order to compute $\mu\big(A \triangle (C'\cdot H)\big)$, we need only 
consider 
\[ \mu\big((C'\cdot H)  \setminus A\big)=\sum\limits_{c\in C'} \mu\big((c\cdot 
H)\setminus 
A\big) .\]
\noindent For $c$ in $C'$, note that the definable set $c\inv ((c\cdot 
H)\setminus A)$ equals $H\setminus (c\inv\cdot A)$, which is contained in 
$A\cdot A\inv$, and hence has a (real-valued) measure. Since $\mu(H)>0$, 
the measure $\mu$ 
normalised by $\mu(H)$ induces a left-invariant Keisler measure on the 
definable subsets of $H$ of the form $X\cap H$, where $X$ is 
$\varphi$-definable over $M$.  By \cite[Theorem 2.3 (vi)]{CPT19}, such a 
measure is unique and furthermore wide sets are exactly the generic sets (\cf 
Fact \ref{F:HrPi}). By construction of $H_\varphi^0$, we conclude that 
$\mu(H\setminus c\inv \cdot A)=0$ for $c$ in $C'$, so 
\[ 
\mu\big(A \triangle (C'\cdot H)\big)=   \mu\big(A \setminus (C'\cdot H)\big) + 
\mu\big((C'\cdot H)  \setminus A\big) =0 <\frac{\epsilon}{2}\cdot\mu(H).\]
However, for $n\ge |C|$ sufficiently large, we conclude by \L o\'s's theorem 
that the 
corresponding trace $A_{n}$ is contained in $C(A_n)\cdot H(G_{n})$ and 
\[\big|A_n\triangle \big(C'(A_n)\cdot H(G_n) \big)\big|<\epsilon |H(G_n)| .\] 

\noindent This yields the desired contradiction.
\end{proof}

We now turn to the proof of Theorem B in the introduction. First, we note that 
a straightforward compactness argument yields
(non-quantitative) bounds on the complexity of the representation of a weakly
normal subset as a boolean combination of suitably chosen subgroups. The fact that the subgroups in Proposition \ref{P:wn} below can be expressed as bounded boolean combinations of translates of $A$ goes beyond \cite[Theorem 1.4]{tS18}.

\begin{prop}\label{P:wn}
Given natural numbers $r$, $k$, and $l$, there are natural numbers
$n=n(r,k,l)$ , $m=m(r,k,l)$ and $t=t(r,k,l)$ such that for any abelian
group $G$ and any
subset $A\subseteq G$ with an $(r,k,l)$-weakly normal representation, the set $A$ is a
boolean combination of complexity at most $t$ of cosets of subgroups $H_1,\ldots, H_n$
of $G$. Moreover, each subgroup $H_i$ is a boolean combination of complexity at most $m$ of
translates of $A$.
\end{prop}
\begin{proof}
Otherwise, as in the proof of Theorem \ref{T:smalltripling}, assume that, for
fixed positive integers $r$, $k$ and $l$, there are $n$, $m$ and $t$ in $\N$, and an
abelian group $G_{n,m,t}$ and a subset
$A_{n,m,t}$ such that $A_{n,m,t}$ admits  an
$(r,k,l)$-weakly normal representation, but it is not a boolean combination of
complexity at most $t$ of cosets of $n$ subgroups of $G$, each of which is
a boolean combination of complexity at most $m$ of
translates of the set $A_{n,m,t}$.

Set  $G_n=G_{n,n,n}$ and $A_n=A_{n,n,n}$. Since each $A_n$ has an
$(r,k,l)$-weakly
normal representation, there are $r$-weakly normal
subsets
$B_{n,1},\ldots,B_{n,k}$ and
$C_{n,1},\ldots,C_{n,l}$ such that
\[
A_n = \bigcup\limits_{j=1}^k B_{n,j} \cap \bigcup\limits_{i=1}^l G_n\setminus
C_{n,i} .\]

We consider an expansion $\LL'$ of $\LL$  with $k+l$
new predicates and regard each group $G_n$ as an $\LL'$-structure $M_n$, where
the predicates are interpreted as the sets $B_{n,1},\ldots,B_{n,k},
C_{n,1},\ldots,C_{n,l}$.
Choose a non-principal ultrafilter $\mathcal U$ on $\N$, and
consider the ultraproduct
$G=\prod_{\mathcal U} G_n$.  The definable set $A=\prod_{\mathcal
	U}  	A_n$ admits an $(r,k,l)$-weakly normal representation,  by \L
o\'s's theorem, thus the definable set $A$ is a boolean
combination of $r$-weakly normal formulae.

By Corollary \ref{C:BC}, the definable set $A$ is a boolean combination of
complexity at most $t_0$ of $n_0$ definable subgroups, each of which is
itself a boolean combination of complexity at most $m_0$ of translates of $A$.

\L o\'s's theorem gives the desired contradiction, by choosing $n\geq
n_0+m_0+t_0$ sufficiently large.
\end{proof}

\begin{theorem}\label{T:wn}
	Given natural numbers $r$, $k$, and $l$, there are natural numbers
	$n=n(r,k,l)$ and $m=m(r,k,l)$ such that for any abelian
	group $G$ and any
	subset $A\subseteq G$ with an $(r,k,l)$-weakly normal representation, there are
	subgroups $H_1,\ldots, H_n$
	of $G$, each contained in $A-A$, with \[A\subseteq \bigcup\limits_{i=1}^{n}
	g_i+H_i,\]
	for some $g_1,\ldots, g_n$ in $A$. Furthermore, each $H_i$ is a boolean
	combination of complexity at most $m$ of translates of $A$.

\end{theorem}

\begin{proof}
As in the proof of Proposition \ref{P:wn}, if the statement does not hold,
there are
fixed  integers $r$, $k$ and $l$ such that for each $n$ and $m$ in $\N$, we
find an abelian group
$G_{n,m}$ and a subset
	$A_{n,m}$ such that $A_{n,m}$ admits  an
	$(r,k,l)$-weakly normal representation, yet \[ A_{n,m}\not\subseteq
	\bigcup\limits_{i=1}^{n}
	g_i+H_i,\]
	for any $g_1,\ldots, g_n$ in $A_{n,m}$, for all subgroups $H_1,\ldots, H_n$
	of $G_{n,m}$, each contained in $A_{n,m}-A_{n,m}$, which are boolean
	combinations of complexity at most $m$ of translates of $A_{n,m}$.

	Set $G_n=G_{n,n}$ and $A_n=A_{n,n}$. Choose a non-principal ultrafilter
	$\mathcal U$ on $\N$, and
	consider the ultraproduct
	$G=\prod_{\mathcal U} G_n$.  Choose a sufficiently saturated elementary
	extension
	$\UU$ of the
	model
	$M=\prod_{\mathcal U} M_n$.  As observed before,  the definable set
	$A=\prod_{\mathcal
		U}  	A_n$ admits an $(r,k,l)$-weakly normal representation,  by \L
	o\'s's theorem, so the formula $\varphi(x,y)=A(x+y)$ is a boolean
	combination of $r$-weakly normal formulae.

	Let $\mathcal F$ be the family of cosets $m+H$, where $m$ belongs to 
	$A(M)$ and the definable subgroup  $H\subseteq
	A-A$ over $M$ is given by a finite boolean combination of
	translates of $A$. By \L o\'s's theorem, no finite collection of $\mathcal 
	F$
	covers $A$: otherwise, considering the traces of these subgroups in the
	corresponding $G_{n,n}$ for sufficiently large $n$, we would have that
	$A_{n,n}$ is covered by a finite union of translates of these subgroups,
	which have bounded complexity as boolean combinations of translates of
	$A_{n,n}$.

By compactness, we obtain a complete $\varphi$-type $p(x)$ over $M$
	containing the formula $A(x)=\varphi(x,\mathrm{id}_G)$ such that the 
	partial type \[ p(x)\cup\{\neg(m+H)(x)\}_{m+H\in \mathcal F}, \]
	 is
	consistent. By Remark
	\ref{R:stab_def}, the
	subgroup $\stab_\varphi(p)$ is $\varphi$-definable, so it is a boolean
	combination of translates of $A$.  Clearly
	$\stab_\varphi(p)
	\subseteq A-A$, by construction. Lemma \ref{L:wn_stab} yields that the type $p$
	contains the definable set
	$\big(A\cap(m+\stab_\varphi(p))\big)(x)$, for some $m$ in $M$. In
	particular, the $M$-definable set $\big(A\cap(m+\stab_\varphi(p))\big)(x)$
	is non-empty, so there is a realisation in $M$, since $M$ is an elementary
	substructure. Replacing the element $m$, we may assume that it lies in
	$A(M)$. Hence, the coset $m+\stab_\varphi(p)$ belongs to the family
	$\mathcal F$, which contradicts our choice of $p$ and hence implies the result.
\end{proof}

As an immediate consequence we obtain the result for finite sets stated as Theorem B in the introduction.
An analogous result holds  whenever $G$ carries a finitely additive probability
measure $\mu$ on the boolean algebra of translates of $A$ with $\mu(A)>0$.


\begin{thebibliography}{99}


\bibitem{eC11} E. Casanovas, \emph{Simple theories and hyperimaginaries},
Lecture Notes in Logic {\bf 39}, Association Symb. Logic, Chicago, IL;
Cambridge University Press, Cambridge, (2011).


\bibitem{CPT19} G. Conant, A. Pillay and C. Terry. \emph{A group version of
stable regularity}, Math. Proc. Camb. Philos. Soc. {\bf 168} (2) (2020), 405--413.

\bibitem{gC20} G. Conant, \emph{On finite sets of small tripling or small alternation in arbitrary groups}, Combinatorics, Probability and Computing {\bf 29} (2020), 807--829.

\bibitem{gC21} G. Conant, \emph{Quantitative structure of stable sets in arbitrary finite groups}, Proc. Amer. Math. Soc., (2021). \url{http://doi.org/10.1090/proc/15479}

\bibitem{EL14}C. Even-Zohar and S. Lovett, \emph{The Freiman-Ruzsa theorem over 
finite fields.} J. Combin. Theory Ser. A {\bf 125} (2014), 333--341.

\bibitem{gF73} G. A. Fre\u{\i}man, \emph{Foundations of a structural theory
	of set addition (Translated from the Russian)}, Transl. Math. Monographs,
	{\bf
	37}, AMS, Providence, R. I. (1973).

\bibitem{GR07} B. Green and I. Z. Ruzsa, \emph{Fre\u{\i}man's theorem in an
arbitrary
	abelian group}, J. Lond. Math. Soc. {\bf  75} (2007), 163--175.

\bibitem{GS08} B. Green and T. Sanders, \emph{A quantitative version of the idempotent theorem in harmonic analysis},
Ann. Math. (2) {\bf 168} (2008), 1025--1054.

\bibitem{HP87} E. Hrushovski and A. Pillay. \emph{Weakly normal groups}, in
Logic
Colloquium 85, Stud. Logic Found. Math. {\bf 122}, North-Holland, Amsterdam
(1987),  233--244.

\bibitem{HP94} E. Hrushovski and A. Pillay. \emph{Groups definable in local
fields
and pseudo-finite fields}. Israel J. Math. {\bf 85} (1994),  203--262.

\bibitem{eH12} E. Hrushovski, \emph{Stable group theory and approximate
	subgroups}, J. Amer. Math. Soc. {\bf 25} (2012), 189--243.

\bibitem{jK87} H. J. Keisler, \emph{Measures and forking}, Annals Pure Applied Logic {\bf
34} (1987), 119--169.

\bibitem{MOS18} S. Montenegro, A. Onshuus and P. Simon, \emph{ Groups with
f-generics in NTP$2$ and PRC fields}, J. Inst. Math. 
Jussieu {\bf 19} (3) (2019), 821--853.

\bibitem{dP19} D. Palac\'in, \emph{On compactifications and product-free sets},  to
appear in J.  London Math. Soc. {\bf 101} (1) (2020), 156--174.

\bibitem{aP96} A. Pillay, \emph{Geometric stability theory},
Oxford Logic Guides {\bf 32}, Oxford Science Pub.,The Clarendon Press, Oxford
University Press, New York (1996).

\bibitem{aP02} A. Pillay, \emph{Model-theoretic consequences of a theorem of
Campana and Fujiki}, Fund. Math. {\bf 174} (2002), 187--192.

\bibitem{iR94} I. Z. Ruzsa, \emph{Generalized arithmetical progressions and 
sumsets}, Acta Math. Hungar. {\bf 65} (1994), 379--388.

\bibitem{iR99} I. Z. Ruzsa, \emph{An analog of Fre\u{\i}man's theorem in
groups},
Ast\'erisque {\bf 258} (1999), 323--326.


\bibitem{tS12} T. Sanders, \emph{On the Bogolyubov-Ruzsa lemma}, Anal. PDE {\bf 
5} (2012), 627--655.

\bibitem{tS18} T. Sanders, \emph{The coset and stability rings}, Online J. Anal. Comb. {\bf 15} 1 
(2020).

\bibitem{tSch11} T. Schoen, \emph{Near optimal bounds in Fre\u{\i}man's 
theorem}, Duke Math. J. {\bf 158} (2011), 1--12.

\bibitem{oS18} O. Sisask, \emph{Convolutions of sets with bounded VC-dimension are uniformly continuous}, Discrete Anal. 2021:1.

\bibitem{sS17} S. Starchenko, \emph{N{IP}, {K}eisler measures and 
combinatorics}, in S\'{e}minaire Bourbaki, Ast\'{e}risque {\bf 390} (2017), 
303--334. 

\bibitem{TV06} T. Tao and V. Vu, \emph{Additive Combinatorics}, Cambridge Studies in Advanced Mathematics {\bf 105}, Cambridge University Press, Cambridge (2006).

\bibitem{TZ12} K. Tent and M. Ziegler, \emph{A course in model theory}, Lecture
Notes in Logic {\bf  40}, Association Symb. Logic, La Jolla, CA; Cambridge
University Press, Cambridge (2012).

\bibitem{TW18} C. Terry and J. Wolf. \emph{Quantitative structure of stable
	sets in finite abelian groups}, Trans. Amer. Math. Soc. {\bf 373} (2020), 3885--3903.

\bibitem{TW19} C. Terry and J. Wolf. \emph{Stable arithmetic regularity lemma
in the finite-field model}. Bull. London Math. Soc. {\bf 51} (2019), 70--88.




\end{thebibliography}
\end{document}